\documentclass[11pt]{article}

\usepackage[english]{babel}
\usepackage{geometry}
\usepackage{xcolor}
\usepackage{amsmath}
\usepackage{comment}
\usepackage{amssymb}
\usepackage{mathtools}
\usepackage[utf8]{inputenc}
\usepackage{color}
\usepackage{amsthm}

\usepackage{tikz}
\usetikzlibrary{graphs,arrows.meta}
\usepackage[shortlabels]{enumitem}

\usepackage{graphicx}
\usepackage{caption}
\usepackage{subcaption}
\usepackage{hyperref}
\usepackage{xcolor}
\usepackage{xr}
\usepackage[capitalise,nameinlink]{cleveref}
\usepackage{float}
\usepackage{soul} %  \st{} w

\usepackage{ytableau}

\usepackage{multirow} 

\newcommand{\vanish}[1]{}

\makeatletter

\providecommand*{\bigcupdot}{%
  \mathop{%
    \vphantom{\bigcup}%
    \mathpalette\@bigcupdot{}%
  }%
}
\newcommand*{\@bigcupdot}[2]{%
  \ooalign{%
    $\m@th#1\bigcup$\cr
    \sbox0{$#1\bigcup$}%
    \dimen@=\ht0 %
    \advance\dimen@ by -\dp0 %
    \sbox0{\scalebox{2}{$\m@th#1\cdot$}}%
    \advance\dimen@ by -\ht0 %
    \dimen@=.5\dimen@
    \hidewidth\raise\dimen@\box0\hidewidth
  }%
}
\makeatother

\newcommand{\bsy}{\boldsymbol}

\theoremstyle{definition}
\newtheorem{definition}{Definition}

\newtheorem{theorem}{Theorem}
\newtheorem{corollary}[theorem]{Corollary}

\newtheorem{example}{Example}

\newcommand{\vt}{vacillating tableau}
\newcommand{\vtx}{vacillating tableaux}
\makeatletter

\title{New Approaches to Identities for Vacillating Tableaux } 
\author{Derek Wu\footnote{dwu120@tamu.edu} and Catherine H. Yan\footnote{huafei-yan@tamu.edu. Yan was supported  by the Simons Collaboration Grant for Mathematics 704276. } 
\\
Department of Mathematics, Texas A\&M University, \\
College Station, TX 77843}

\date{} % May 2025 

\begin{document}

\maketitle 

\begin{abstract}

A fundamental identity for the number of vacillating tableaux was originally obtained through the representation theory of partition algebras. We extend this identity to arbitrary differential posets and show that it, together with several analogous identities, follows directly from the structural properties of differential posets. Specializing to Young's lattice and its Cartesian powers, we further obtain new bijective proofs via a simple deletion--insertion process.
\end{abstract}

% MSC numbers: 05A17, 05A19, 05E10 
% Key words:  vacillating tableau, differential poset 
\section{Introduction}

One of the fundamental identities in the representation theory of partition algebras, due to Halverson and Lewandowski \cite{HL05}, is 
\begin{equation} \label{eq: Identity1} 
    n^k=\sum_{\lambda\vdash n} f^\lambda m_{k}(\lambda),  \qquad \text{for } n \geq 1, 
\end{equation}
 where the sum is over integer partitions $\lambda$ of $n$, $f^\lambda$ is the number of standard Young tableaux (SYT) of shape $\lambda$, and $m_k(\lambda)$ is the number of vacillating tableaux of shape $\lambda \vdash n$ and length $2k$. Here a \emph{\vt\  of shape $\lambda$ and length $2k$}  is a sequence of partitions 
$\lambda_0=(n) \supseteq \lambda_{\frac12} \subseteq \lambda_1 
\supseteq \lambda_{\frac32} \subseteq \lambda_2 
\supseteq 
\cdots 
\subseteq \lambda_{k-1} 
\supseteq \lambda_{k-\frac{1}{2}} \subseteq \lambda_k  =\lambda 
$ such that the sizes of consecutive partitions differ by exactly one. 
Following the English notation, we represent an integer partition $\lambda=(a_1, a_2, \dots, a_k)_>$ by the Young diagram that contains $a_j$ squares in the $j$-th row, which are aligned in the upper-left corner. Below is an example of a vacillating tableau of shape $\lambda=(3,2,1) \vdash 6$ and length $6$. 
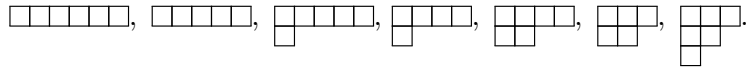
\[ 
\ytableausetup{boxsize=7pt}  
\ydiagram{6}, \  \ydiagram{5},\ 
\ydiagram{5,1},\, \ydiagram{4,1}, \ \ydiagram{4,2},\ \ydiagram{3,2},\ \ydiagram{3,2,1}.
\]

In literature, there are two  definitions of a \vt. 
The  definition  above was introduced in \cite{HL05}, while the term \emph{\vt} was adopted from \cite{CDDSY}, where    the combinatorial structures of crossings and nestings in matchings and set partitions were characterized. In \cite{CDDSY} Chen et al.~defined  a \vt\ as a  sequence of integer partitions that starts from the empty shape $\emptyset$, where each $\lambda_i$ is obtained from $\lambda_{i-\frac12}$ by adding a square, deleting a square, or doing nothing, depending on whether $i$ is an integer or a half integer.  When $n \geq 2k$, the two notions are equivalent.  

In this paper we use the definition from \cite{HL05} as given in the first paragraph, while the notion defined in \cite{CDDSY} will be referred to as a \emph{simplified \vt}. 
Another way to understand \vtx\ is to view  them as  walks in Young's lattice $\mathbb{Y}$, the poset of  all integer partitions ordered by inclusion.  
 Explicitly,   a \vt\  is a walk in the Hasse diagram of $\mathbb{Y}$ starting  from $(n)$ and using alternating down and up steps.

 %In differential posets, there are down-up alternating walks. 

Identity \eqref{eq: Identity1} presents a combinatorial analogue of the Schur-Weyl duality between the symmetric group algebra and the partition algebra. 
In \cite{HL05}, 
Halverson and Lewandowski constructed an elegant combinatorial proof.  
Using a deletion-insertion algorithm based on jeu de taquin (\texttt{jdt}) and 
the Robinson-Schensted-Knuth (RSK) insertion algorithm, they  associate to each integer sequence $\bsy{i} \in [n]^k$, where $[n]\coloneqq\{ 1, 2, \ldots, n\}$,  a pair of tableaux of the same shape, one being a standard Young tableau and the other a vacillating tableau. %The explicit description of the bijection $DI_n^k$ is recalled in Section 3. 

 Let $\mu$ be any fixed partition of $n$. 
Kratthenthaler \cite{Kra23} generalized 
Identity \eqref{eq: Identity1}  to the following form, 
%using the growth diagrams developed by Fomin  (\cite{Fomin95, Fomin95-2, Fomin88}):
\begin{eqnarray} \label{Eq:Id2} 
    f^\mu n^k = \sum_{\lambda \vdash n} f^\lambda \  m_k(\mu \rightarrow \lambda),
\end{eqnarray}
where $m_k(\mu \rightarrow \lambda)$ is the number of  vacillating tableaux of length $2k$ from $\mu$ to $\lambda$, that is, sequences of partitions $\lambda_0=\mu \supseteq \lambda_{\frac12} \subseteq \lambda_1 
\supseteq \lambda_{\frac32} \subseteq \lambda_2 
\supseteq 
\cdots 
\subseteq \lambda_{k-1} 
\supseteq \lambda_{k-\frac{1}{2}} \subseteq \lambda_k  =\lambda 
$ such that consecutive partitions differ by exactly one square. 
Clearly $m_k(\lambda) = m_k((n) \rightarrow \lambda)$. 
The proofs in \cite{Kra23} are presented  in terms of Fomin's growth diagrams   (\cite{Fomin95, Fomin95-2, Fomin88}).

In this paper,  we  consider  Identities \eqref{eq: Identity1}  and \eqref{Eq:Id2} in the content of differential posets,  which  contains Young's lattice as a prototypical example.
More precisely, we study  certain alternating walks in the Hasse diagram of a  differential poset and 
prove that Identity \eqref{Eq:Id2}, together with several of its analogs, follows 
from the structural property of differential posets by elementary combinatorial argument. 
We then present a new deletion--insertion algorithm that is simpler than the construction of Halverson and Lewandowski. Finally, we extend this new algorithm to obtain bijective proofs of these variant identities in Young's lattice $\mathbb{Y}$ and its Cartesian powers $\mathbb{Y}^r$.

\section{Walks in Differential Posets} 

%Define notations in differential posets. 

Differential posets were introduced by Stanley \cite{Stanley88} in 1988. They are a class of partially ordered sets with remarkable combinatorial and algebraic properties.

\begin{definition}[Stanley \cite{Stanley88}]
    Let $r$ be a positive integer. A poset $P$ is called \emph{$r$-differential} if it satisfies the following three conditions. 
    \begin{itemize}
        \item[(D1)] $P$ is locally finite, graded, and has a $\hat 0$ element. 
        \item[(D2)] If $x \neq y$ in $P$ and there are exactly $k$ elements of $P$ which are covered by both $x$ and $y$, then there are exactly $k$ elements of $P$ which cover both $x$ and $y$. 
        \item[(D3)] If $x\in P$ and $x$ covers exactly $k$ elements of $P$, then $x$ is covered by exactly $k+r$ elements of $P$. 
    \end{itemize}
\end{definition} 

A typical example of a differential poset is 
Young's lattice $\mathbb{Y}$, which is $1$-differential. In general, $\mathbb{Y}^r$ is $r$-differential. There are infinitely many  irreducible  $r$-differential posets; see the discussion in \cite[Section 6]{Stanley88}. Among them, another one with significant combinatorial properties is 
the Fibonacci differential poset $Z(r)$ described in \cite[Section 5.2]{Stanley88}. 
% In this paper, we are only concerned with the case $r=1$, and we simply refer to a $1$-differential poset as \emph{differential poset}. 

Let $P$ be a $r$-differential poset with the rank function $\rho$ and  $P_n=\{x  \in P: \ \rho(x)=n\}$ be the set of elements of rank $n$. 
Given $x \in P_n$, define 
\[
C^-(x)=\{ y \in P:\ x \text{ covers } y\}, \qquad 
C^+(x)=\{y \in P:\ y \text{ covers } x\}.
\]
Note that both $C^-(x)$ and $C^+(x)$ are finite. 

Let $K$ be a field and $KP$ be the $K$-vector space with basis $P$, while $\hat K P$ be the $K$-vector space of arbitrary linear combination of elements of $P$. We can define two continuous linear
operators, $U$ and $D$, from $\hat K P \rightarrow \hat K P$, by 
\[
U x = \sum_{y \in C^+(x)} y, \qquad D x = \sum_{y \in C^-(x)} y,
\]
for all $x \in P$. 
We say that $U$ is the \emph{up} operator and $D$ is the \emph{down} operator. 
In addition, let $I: \hat KP\rightarrow \hat KP$ be the identity transformation. 
A fundamental result proved in \cite[Theorem 2.2]{Stanley88} asserts that 
if $P$ is $r$-differential, then  
\[
DU-UD=rI,
\]
where linear operators act right-to-left, e.g., $DU x = D(Ux)$. 

There is a symmetric bilinear form $\langle \  ,\ \rangle:  KP \times KP \rightarrow K$, defined by  
\[
\langle x, y\rangle =\delta_{x,y} \qquad \text{for 
$x,y \in P$}.
\]
The form, as introduced in \cite{Stanley88}, was defined on $\hat K P \times KP$ and continuous in the first coordinate. Here we only apply it to finite linear combinations of elements in $P$, hence 
$KP \times KP$. It follows that 
\[
\langle \sum a_x x , \sum b_x x \rangle = \sum a_x b_x, 
\]
when the variable $x$ ranges over a finite subset of $P$.  The operators $U$ and $D$ are adjoint under this form, i.e., 
\[
\langle Dx,y \rangle  = \langle  x, U y \rangle, \qquad 
\langle Ux, y \rangle = \langle x, Dy \rangle.  
\]

The problem of counting walks in the Hasse diagram of a differential poset was initially studied in \cite{Stanley88}. In particular, various generating functions and formulas  were given for walks following a special pattern $w \in \{U, D\}^*$. 
The main goal of the present work is to use these ideas to extend Identities \eqref{eq: Identity1} and \eqref{Eq:Id2} to differential posets. 
The key ingredient of our approach is the following well-known  identity that expands $(UD)^k$ as a linear combination of $U^iD^j$. 
\begin{theorem} \label{thm:UD}
Let $P$ be a $r$-differential poset. Then for any $k \geq 0$,  we have 
\begin{eqnarray} \label{eq:(UD)^n}
(UD)^k =\sum_l r^{k-l} S(k,l) U^l D^l,  
\end{eqnarray}
where $S(k,l)$ is the  Stirling number of the second kind that counts the number of set partitions of $[k]$ into $l$ blocks. 
\end{theorem}
This identity  appears,  for example, in \cite[Proposition 4.9]{Stanley88}; earlier references trace back to Commet’s book  \cite[Exercise 2, p.220]{Commet74}.
It is also straightforward to prove by induction. Below we provide a simple combinatorial argument, which shows explicitly the connection to set partitions. 

For a partition $\pi$ of the set $[n]$, we write Arc$(\pi)$  for the set of pairs
of integers $(i,j)$ such that $i$ and $j$ are consecutive elements in the same block of $\pi$. One often represents set partitions by drawing a graph whose vertex set is $[n]$ and whose edge set is Arc($\pi$). Such a graph is called the standard representation of $\pi$. For example,   the standard representation  of $\pi=\{\{1, 3, 4,7 \}, \{2,5\}, \{6\}\}$  is 

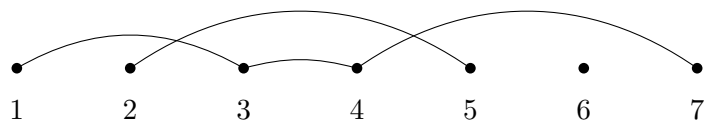
\begin{center} 
\begin{tikzpicture}[scale=1,
    every node/.style={circle, draw, minimum size=8mm, inner sep=0pt}
]

% Vertices as small black circles with labels underneath
\foreach \x/\label in {
    0/1,
    1.5/2,
    3/3,
    4.5/4,
    6/5,
    7.5/6,
    9/7
}{
    \fill (\x,0) circle (2pt);
    \node[below=4pt, draw=none] at (\x,0) {\label};
}

% Upper-half edges
\draw (0,0) to[bend left=30] (3,0);     % (1,3)
\draw (3,0) to[bend left=15] (4.5,0);   % (3,4)
\draw (4.5,0) to[bend left=35] (9,0);   % (4,7)
\draw (1.5,0) to[bend left=35] (6,0);   % (2,5)

\end{tikzpicture}
\end{center}

\begin{proof}[Proof of Theorem 1.]
 The case with $k=0$ is trivial. Assume $k \geq 1$. 
    Write $(UD)^k$ as $U_1D_1U_2D_2 \cdots U_kD_k$, where $U_i=U$ and $D_i=D$ for all $i$. We push all the $D_i$-operators, in the order 
    $i=k-1, \dots, 1$,  to the
    right of all $U$-operators using the relation $DU=UD+rI$. 
    The relation implies that, while we try to push the operators $D_i$ to the right of  $U_j$ with $i <j$,  either
    we get $U_jD_i$, or this pair of operators cancel  each other and we get a constant factor $r$. 

    The coefficient of $U^lD^l$ at the end of this process is the number of ways we can push exactly $l$ many $D$-operators to the right of  all $U$'s. Hence $k-l$ of the $D$-operators are canceled with some $U$-operators.   Each such a process defines a unique set partition $\pi$ on $[k]$ as follows: $(i,j)$ is in Arc($\pi$)  if and only if $D_i$ cancels with $U_j$.  This gives the coefficient $r^{k-l} S(k,l)$.     
\end{proof}

A similar argument leads to the following expression, which will be used in our calculations later. 
\begin{corollary} \label{cor:DU} 
 In a $r$-differential poset, 
    we have 
    \[
    D^n U^l = \sum_i r^i \binom{n}{i} \binom{l}{i} i!  U^{l-i}D^{n-i},
    \]
    where $i \leq \min(n, l)$. 
\end{corollary}
 
\Cref{thm:UD} and \Cref{cor:DU} have appeared in \cite{Fomin95-2} for dual graded graphs, which  generalize  differential posets. The 
results in \cite{Fomin95-2} use  the notion of \emph{$r$-colored diagonal set}, another way to encode the cancellation of the $D$ and $U$ operators. The identities below therefore follow from the general theory developed in \cite{Fomin95-2}. Here, however, we show that they  can also be proved  easily  by elementary combinatorial arguments.

\subsection{Down-up alternating walks} 
In this paper,% we are only concerned with the case $r=1$, and we simply refer to a $1$-differential poset as \emph{differential poset}. 
we are interested in walks in an $r$-differential poset $P$. 
Given $x \in P_n$ and a word $w \in \{U, D\}^*$, we have $wx=\sum_{y} c_y y$, where $c_y$ is the number of walks from $x$ to $y$ in the Hasse diagram of $P$ of length $|w|$, following the pattern of $w$, in which $U$ means an up  step and $D$ means a down step with respect to the rank of $P$, and one should read the letters in the word $w$ from right to left.  Call such walks  \emph{$w$-walks}. 
In particular, we say that a $w$-walk with $w=(UD)^k$ is a \emph{down-up alternating  walk} of length $2k$. 
Explicitly, a down-up alternating  walk from $x$ to $y$ of length $2k$ is a sequence   $x=y_0 \geq y_{\frac{1}{2}} \leq y_1 \geq y_{\frac{3}{2}} \leq y_2 \geq \cdots  \geq y_{k-\frac12} \leq y_k =\lambda$ such that any two consecutive elements are adjacent in the Hasse diagram of $P$. If $x \in P_n$, then clearly 
$y_i \in P_n$ and $y_{i+\frac12} \in P_{n-1}$ for any integer $i$. 
We denote by $m_k(x \rightarrow y)$ the number of down-up alternating   walks from $x$ to $y$ and of length $2k$. 

Let $e(x)$ be the number of saturated chains from $\hat 0$ to $x$. In other words, $e(x)$ is the $U^n$-walk from $\hat 0$  to $x$ if $x \in P_n$,
which can also be expressed as  
$e(x)= \langle D^n x, \hat 0\rangle$.

\begin{theorem} \label{thm:UD1} 
Let $P$ be a $r$-differential poset and  $n, k \in \mathbb{N}$. Then for $x \in P_n$,  we have the identity 
\begin{eqnarray} \label{eq:UD-even} 
e(x)  (rn)^k = \sum_{y \in P_n} m_{k}(x \rightarrow y) e(y).
\end{eqnarray} 
\end{theorem}
\begin{proof}  
Applying $(UD)^k$ to $x \in P_n$, we obtain 
\[
(UD)^k x = \sum_{y \in P_n} m_k(x \rightarrow y)  \ y. 
\]
Hence 
\[
 D^n (UD)^k x = \sum_{y \in P_n} m_k(x \rightarrow y)  e( y) \ \hat 0,
\]
 and the right hand side of \eqref{eq:UD-even} can be expressed as 
\[
\langle D^n (UD)^k x,\hat 0  \rangle.  
\]
Using \Cref{thm:UD} and \Cref{cor:DU}, we have
\begin{eqnarray} \label{eq:UD-full}
 D^n (UD)^k  = D^n \sum_l r^{k-l} S(k,l) U^l D^l,    \nonumber \\
 =  \sum_{l} r^{k-l}S(k,l) \sum_{i\leq  l } r^i\binom{n}{i} \binom{l}{i} i! U^{l-i} D^{l+n-i}. 
\end{eqnarray}
Apply \eqref{eq:UD-full} to $x\in P_n$. The only way to get $\hat 0$ is when $i=l$, where the coefficient is simplified to 
\[
e(x) r^k \sum_{l} S(k,l) (n)_l = e(x) r^k n^k,
\]
where $(t)_l=t(t-1)\cdots (t-l+1)$ is the $l$-th falling factorial of $t$. 
The last equation is obtained by applying 
the basic polynomial identity $t^k=\sum_{l} S(n,l) (t)_l$
for Stirling numbers of the second kind; See e.g. \cite[Chapter 1.4]{EC2}.) 
\end{proof}

In Young's lattice, when $x$ is the integer partition $\lambda$, we have  $r=1$  and $e(\lambda) = f^\lambda$, the number of standard Young tableaux of shape $\lambda$. 
Hence Identity \eqref{Eq:Id2} is a special case of \eqref{eq:UD-even}.

%
%\begin{corollary}  
%\begin{eqnarray}
%   n^k f^\mu =\sum_{\lambda\vdash n} f^\lambda m_{k}(\mu \rightarrow \lambda),   \label{id1:general} 
%\end{eqnarray}
%where $m_k(\mu \rightarrow \lambda)$ is the number of vacillating tableaux from $\mu$ to $\lambda$ of length $2k$. 
%\end{corollary} 

Similarly,  one can consider down-up alternating walks of odd length $2k+1$. That is, 
a sequence $x=y_0 \geq y_{\frac{1}{2}} \leq y_1 \geq y_{\frac{3}{2}} \leq y_2 \geq \cdots  \geq y_{k-\frac12} \leq y_k \geq y_{k+\frac12}=\lambda$ such that $
y_i \in P_n$ and $y_{i+\frac12} \in P_{n-1}$ for any integer $i$. 
Denote by  $m_{k+\frac12}(x \rightarrow y)$ the number of down-up alternating walks of length $2k+1$.
\begin{corollary} \label{thm:UD2}
Let $P$ be a $r$-differential poset and  $n, k \in \mathbb{N}$ with $n \geq 1$. Then for $x \in P_n$, we have the identity 
\begin{eqnarray} \label{eq:UD-odd}
e(x)  (rn)^k= \sum_{y \in P_{n-1} } m_{k+\frac{1}{2}} (x \rightarrow y)  e(y) 
\end{eqnarray}
\end{corollary} 
\begin{proof}
Observe that the right-hand side of \eqref{eq:UD-odd} is exactly the number of $w$-walks from $x$ to $\hat 0$, where $w=D^{n-1} \big( D(UD)^k\big) = D^n (UD)^k$, which by \Cref{thm:UD1} is $e(x) (rn)^k$.    
%    Applying $D(UD)^k$ to $x\in P_n$, we obtain 
 %   \[
  %  D(UD)^n (x) = \sum_{y \in P_{n-1}} m_{k+\frac12} (x \rightarrow y) \ y 
  %  \]
  % Hence 
  % \[
  % \sum_{y \in P_{n-1}} m_{k +\frac12}(x \rightarrow y) \ e(y) = \langle  D^{n-1} ( D (UD)^k x, \hat 0 \rangle 
 %   = \langle D^n (UD)^k x, \hat 0 \rangle  = e(x) \ n^k. 
 %  \] 
\end{proof}

\subsection{Up-down alternating walks}

In this subsection we consider a variant  of the down-up alternating walks that  begins with an  up step. 
Explicitly,  an \emph{up-down alternating  walk}  from $x \in P_n$ in an 
$r$-differential poset $P$ is a sequence 
$x=y_0 \leq y_{\frac{1}{2}} \geq y_1 \leq y_{\frac{3}{2}} \geq y_2 \geq \cdots $ such that $
y_i \in P_n$ and $y_{i+\frac12} \in P_{n+1}$ for any integer $i$. 

First we give an expansion of the operator $(DU)^k$, which is analogous to Identity \eqref{eq:(UD)^n} and was known in \cite[Sec.3.8]{Fomin95-2}. 
It can be deduced algebraically by computing $(DU)^k$ as $D (UD)^{k-1} U$, using Identity \eqref{eq:(UD)^n}
with $k$ replaced by $k-1$, and the following equations
\[
D^l U = UD^l + rl U^{l-1}, \qquad \qquad  DU^l =U^l D + rlU^{l-1}. 
\]
Below we present a combinatorial argument that is similar to the proof of  
Identity \eqref{eq:(UD)^n}, which reveals the connection to set partitions.

\begin{theorem} Let $k \in \mathbb{N}$. The following identity holds in any $r$-differential poset:
\begin{eqnarray} 
(DU)^k =\sum_l r^{k-l} S(k+1,l+1) U^l D^l.   \label{eq:(DU)^n} 
\end{eqnarray}
\end{theorem} 
\begin{proof}
\iffalse 
Then we have 
\begin{eqnarray*}
   D (UD)^{k-1} U& =&  \sum_l r^{k-1-l} S(k-1, l) DU^{l} D^{l}U  \\ 
     &=& \sum_l r^{k-1-l}S(k-1, l) \big( u^{l+1}D^{l+1} + r(2l+1) U^l D^l + r^2l^2 U^{l-1}D^{l-1} \big).   
\end{eqnarray*}
The coefficient of $U^i D^i$ is 
\begin{eqnarray*} 
& &  r^{k-i}S(k-1, i-1) +r^{k-1-i} r(2i+1) S(k-1, i) +r^{k-1-i-1}r^2(i+1)^2 S(k-1, i+1) \\
&=& r^{k-i}[S(k-1, i-1)+iS(k,i)] + r^{k-i} (i+1) [ S(k-1, i)+(i+1)S(k-1, i+1)] \\ 
&=&  r^{k-i} [S(k,i)+(i+1) S(k, i+1)] \\
&=& r^{k-i} S(k+1, i+1), 
\end{eqnarray*}
which finishes the proof. 
\fi 

The case for $k=0$ is trivial. Assume $k \geq 1$.  
Write $(DU)^k$ as $D_1 U_2 D_2 U_3 \cdots D_k U_{k+1}$ and push the $D$ operators to the right of $U$'s. In the final form each copy of $U^lD^l$ means there is a process with exactly $k-l$  cancellations between the pairs of the form $(D_{i_r}, U_{j_r})$, 
with  $i_r < j_r$ for $r=1, \cdots, k-l$. 
Each such a process corresponds to a unique set partition $\pi$ of $[k+1]$ with $l+1$ blocks, which is given by letting $(i, j)$ be an arc of  $\pi$ whenever $D_i$ cancels with $U_j$. 
\end{proof}
Clearly Identity \eqref{eq:(UD)^n} follows from \eqref{eq:(DU)^n}. In fact, these two identities are equivalent. 
%A combinatorial proof of \eqref{eq:(DU)^n} is presented in \cite{Hariharan}, similar to the one for \eqref{eq:(UD)^n}. 

Let $t_k(x \rightarrow y)$ %(resp. $t_{k+\frac12}(x \rightarrow y)$) 
be the number of up-down alternating walks of length $2k$ %(resp. $2k+1$) 
from $x$ to $y$, where 
 $x , y \in P_n$ for some $n \in \mathbb{N}$. Similarly let  
 $t_{k+\frac12}(x \rightarrow y)$ be the number of such walks of odd length $2k+1$ 
 from $x \in P_n$ to $y \in P_{n+1}$. 
We obtain the following analogous results of \Cref{thm:UD1} and \Cref{thm:UD2}. 

\begin{theorem}
    Let $P$ be an $r$- differential poset and $x \in P_n$. Then we have 
    \begin{eqnarray} \label{eq:DU-even} 
        e(x) r^k(n+1)^k = \sum_{y \in P_n} t_k(x \rightarrow y)  e(y) \qquad \text{ for } n, k \in \mathbb{N}, 
    \end{eqnarray}
    and 
    \begin{eqnarray} \label{eq:DU-odd} 
        e(x) r^{k} (n+1)^{k} = \sum_{y \in P_{n+1}} t_{k-\frac12}(x \rightarrow y)  e(y), \qquad \text{for } n , k \in \mathbb{N} \text{ and } k \geq 1. 
    \end{eqnarray}
\end{theorem}
\begin{proof}
    For \eqref{eq:DU-even}, 
note that $(DU)^k x = \sum_{y \in P_n} t_k(x \rightarrow y)  \ y$. Hence the right-hand side of \eqref{eq:DU-even} is 
$\langle D^n (DU)^k x, \hat 0\rangle$. 
One computes
\begin{eqnarray*}
    D^n (DU)^k x= \sum_{k} r^{k-l}S(k+1, l+1) \sum_{i \leq l} r^i\binom{n}{i} \binom{l}{i} i! U^{l-i} D^{n+l-i} x. 
\end{eqnarray*}
The only way to get $\hat 0$ from the right-hand side of the above equation is when $i=l$, and the coefficient of $\hat 0$ is 
\[
e(x) \sum_{l} r^k S(k+1, l+1) (n)_l 
= \frac{e(x)r^k }{n+1} \sum_l S(k+1, l+1) (n+1)_{l+1} =  \frac{e(x) r^k}{n+1}  (n+1)^{k+1},
\]
which agrees with the left-hand side of \eqref{eq:DU-even}. 

For \eqref{eq:DU-odd}, the right-hand side is the coefficient of $\hat 0$ in $D^{n+1} (U (DU)^{k-1}) x$, which is exactly $D^n (DU)^{k} x$. 
 Hence \eqref{eq:DU-odd} also follows from the above proof.  
\end{proof}

\subsection{Another  application of Theorem 1}

Another fundamental identity proved in \cite{HL05} for vacillating tableaux is 
\begin{eqnarray} \label{Eq:extra}
B_{2k}  = \sum_{\lambda \vdash n} ( m_k(\lambda))^2 = m_{2k}((n)), \qquad \text{for } n \geq 2k, 
\end{eqnarray} 
where $B_m$ is the $m$-th Bell number that counts all set partitions of $[n]$. 
Similar identities involving sums of products of  counts of simplified \vtx\ and their combinatorial interpretations are 
 discussed in \cite{BHPYZ24} and \cite{CDDSY}. 
In this subsection we extend 
Identity  \eqref{Eq:extra} 
 to differential posets. % using \eqref{eq:(UD)^n}. 

Let $x, z \in P_n$ and $k=k_1+k_2$ with $k_1, k_2 \in \mathbb{N}$. Clearly 
\[
m_k(x \rightarrow z) = \sum_{y \in P_n} m_{k_1}(x \rightarrow y) m_{k_2}(y \rightarrow z). 
\]
In particular, when $x=z$, after reversing the walks from $y$ to $z$, we have 
%\tcp{[maybe say after reversing the walks $y\rightarrow x$]}
%\tcy{[Changed]}
\[
\sum_{y \in P_n } m_{k_1}(x \rightarrow y) m_{k_2} (x \rightarrow y) = m_k(x\rightarrow x)=\langle (UD)^k x, x\rangle. 
\]
Then we can apply Identity \eqref{eq:(UD)^n}. 

    If $e(x)=1$, then $U^lD^l x =x$ for $l \leq n$, and $U^lD^l=0$ otherwise. For such an $x$,   we have 
    \[
    m_k(x\rightarrow x)= \sum_{y \in P_n } m_{k_1}(x \rightarrow y) m_{k_2} (x \rightarrow y)  = \sum_{l=1}^n r^{k-l}S(k,l). 
    \]
    This happens when $x=(n)$ or $(1^n)$ in  Young's lattice $\mathbb{Y}$, or the element $11\cdots 1$ in the Fibonacci differential poset $Z(1)$. 
     In particular, the above equation generalizes an identity of Martin and Rollet \cite{MR98}, corresponding to the case $x=(n)$  in $\mathbb{Y}$, for which 
      Krattenthaler     \cite{Kra23} provided a bijective proof.  
    If in addition, $n \geq k$, we obtain 
    \[
    \sum_{y \in P_n } m_{k_1}(x \rightarrow y) m_{k_2} (x \rightarrow y)  = \sum_{l=1}^k r^{k-l} S(k,l)  = B_k(r), 
    \]
    where $B_k(r)$ is the number of $r$-colored set partitions on  $[k]$. Here, an $r$-colored set partition of $[k]$ is a pair $\Lambda=(\pi, \varphi)$, consisting of a set partition $\pi$ of $[k]$ with a map $\varphi: \mathrm{Arc}(\pi) \rightarrow [r]$ labeling its arcs.    
    Identity \eqref{Eq:extra} is the special case with $k_1=k_2$ and $x=(n)\geq 2k_1$ in $\mathbb{Y}$. 

The general case for $r$-differential posets is given by the following theorem. 
\begin{theorem}  Let $P$ be an $r$-differential poset and $x \in P_n$. 
  For $y \leq x$ in $P$, let $e(y,x)$ denote the number of maximal chains of the interval $[y,x]$.  Then for any fixed $k_1, k_2 \in \mathbb{N}$ and $k=k_1+k_2$, 
    \[
    m_k(x \rightarrow x)=\sum_{y \in P_n } m_{k_1}(x \rightarrow y) m_{k_2} (x \rightarrow y) = \sum_{l} r^{k-l} S(k,l) \sum_{y \in P_{n-l}} e(y,x)^2. 
    \]
\end{theorem}

We remark that Identity \eqref{eq:(UD)^n} can also be used to obtain explicit formulas  for special $m_k(\mu \rightarrow \lambda)$ in Young's lattice when $e(y, x)$ is easy to compute. Examples include 
Formulas (1.5) and (1.6) in \cite{Kra23} and the results in Section 4 of \cite{DY25}. 

\section{Algorithmic Bijections  in Young's Lattice} 
The $r$-th powers of Young's lattice,  $\mathbb{Y}^r$, are typical examples of $r$-differential posets and bear special significance in algebraic combinatorics and representation theory.
In the next two sections we present algorithmic bijections for 
the identities in Section 2 in the case $P=\mathbb{Y}^r$. We begin with  the case $r=1$. 

When $P=\mathbb{Y}$  and $x=\mu$ is a shape of size $n$, we have 
$e(\mu)=f^\mu$, and Identities \eqref{eq:UD-even}, \eqref{eq:UD-odd}, \eqref{eq:DU-even}
and \eqref{eq:DU-odd} become  
\begin{eqnarray} \label{eq:Young1}
     n^k f^\mu &=& \sum_{\lambda \vdash n} m_k(\mu \rightarrow \lambda) f^\lambda = \sum_{\lambda \vdash n-1} m_{k+\frac12} (\mu\rightarrow \lambda) f^\lambda, 
     \end{eqnarray} 
     \begin{eqnarray} \label{eq:Young2} 
    (n+1)^k f^\mu  = \sum_{\lambda\vdash n}  t_{k}(\mu \rightarrow\lambda) \ f^\lambda
     = \sum_{\lambda \vdash n+1} t_{k-\frac12}(\mu \rightarrow\lambda) \ f^\lambda ,  \label{eq:Young3}
\end{eqnarray}
%for any $n, k \in \mathbb{N}$, 
where $m_k(\mu \rightarrow \lambda)$ and $m_{k+\frac12}(\mu \rightarrow \lambda)$ count  vacillating tableaux from $\mu$ to $\lambda$ 
with  down-up alternating steps, 
while $t_k(\mu \rightarrow \lambda)$ and $t_{k+\frac12}(\mu \rightarrow \lambda)$ count those with up-down alternating steps.  
%$t_k^\lambda$ is the number of tableau-sequences
%$\lambda_0=(\mu) \subseteq \lambda_{\frac12} \supseteq \lambda_1 
%\subseteq \lambda_{\frac32} \supseteq \lambda_2 
%\subseteq 
%\cdots 
%\supseteq \lambda_{k-1} 
%\subseteq \lambda_{k-\frac{1}{2}} \supseteq \lambda_k  =\lambda 
%$ 
%and consecutive partitions differ by exactly one square. 
%To distinguish with the case before, call such a \vt\ up-down \vt. 

Halverson and Lewandowski gave a bijective proof of Identity \eqref{eq: Identity1} based on a deletion-insertion process, in which the deletion step uses jeu de taquin and the insertion step is given by RSK row insertion.
In this section, we present a simpler deletion--insertion process that avoids the use of jeu de taquin. We then apply this new process to obtain bijective proofs of \eqref{eq:Young1} and \eqref{eq:Young2}. 

We assume that the reader is familiar with the basic definition of RSK insertion on a \emph{partial tableau}, which is a filling of some shape $\lambda$ using distinct integers such that the entries are increasing along each row and each column.   A partial tableau is \emph{standard} if the set of entries is exactly $[n]$, where $n=|\lambda|$. 
%For a full description of these two algorithms, see \cite{HL05}.  

The main step of our bijection is the \emph{lifted insertion}, which inserts an integer 
$i \in [n+1]$ into a SYT  $T$  of size $n$ and yields a STY $S$ of size $n+1$. 

\noindent 
\textbf{Lifted insertion}: Input: a SYT  $T$ of size $n$ and an integer $i$ with $i \in [n+1]$, 
\begin{enumerate} 
\item (Lift): increase every entry $j$ of $T$ satisfying $j \geq i$ by 1; 
\item (Insert): insert $i$ to the tableau obtained in step 1 via  RSK row insertion. 
\end{enumerate}
The output is the SYT $S$.

\medskip 

\noindent 

\begin{proof}[Proof of Identity \eqref{eq:Young1}]
Let $\mu \vdash n$. 
We construct a bijection  $\phi_n^k$ 
\begin{eqnarray} \label{bijection:1}
\phi_n^k: \mathcal{SYT}(\mu) \times [n]^k  \xleftrightarrow{\text{ bijectively } } 
\bigcupdot_{\lambda \vdash n } \mathcal{SYT}(\lambda) \times \mathcal{VT}_{k}(\mu \rightarrow\lambda),
\end{eqnarray}
where $\mathcal{SYT}(\mu)$ is the set of SYTs of shape $\mu$ and 
$\mathcal{VT}_{k}(\mu \rightarrow\lambda)$ is the set of (down-up) vacillating tableaux from $\mu$ to $\lambda$ of length $2k$. 

 Let $T$ be a SYT  of shape $\mu$ and $\{i_1, \dots, i_k\}$ in $[n]^k$  be an integer sequence of length $k$. 
First we define a sequence of SYTs iteratively: 
%\tcp{[Should this be defined iteratively?]} \tcy{[changed]} 
the $0$-th tableau is $T^{(0)}=T$.  Then for integers $j= 1, \dots, k$, 
\begin{enumerate}
    \item $T^{j-\frac12}$ is the SYT of size $n-1$ obtained from $T^{(j-1)}$ by removing the square containing  $n$.  
    \item $T^{(j)}$ is the SYT of size $n$ obtained from $T^{(j-\frac12)}$ by lifted-inserting $i_{j}$. 
\end{enumerate}
%Note that  $T^{(j-\frac12)}$ is  a STY of size $n-1$ and $T^{(j)}$ is a SYT of size $n$.

Let $\lambda_m$ be the shape of $T^{(m)}$ for all integral and half-integral indices $m$, and $\lambda=\lambda_k$ be the shape  of the last tableau $T^{(k)}$. Let $\Gamma$ the sequence  $(\lambda_0, \lambda_\frac12, \lambda_1, \dots, \lambda_k)$. 
Define the image of $(T, \bsy{i})$ under $\phi_n^k$ as 
\[
\phi_n^k(T, \bsy{i}   ) = \left(T^{(k)}, \Gamma \right). 
\]
The image  is clearly in $\mathcal{SYT}(\lambda) \times \mathcal{VT}_k(\mu\rightarrow \lambda)$. 

To show $\phi_n^k$ is a bijection, one notes that the initial tableau $T$ and the integer sequence $\bsy{i}$ can be recovered from the pair $T^{(k)}$ and 
  $\Gamma=(\lambda_0, \lambda_\frac12, \lambda_1, \dots, \lambda_k)$. 
For $j=k, k-1, k-2, \dots, 1$, 
\begin{enumerate}
    \item Given SYT $T^{(j)}$ and  $\lambda_{(j-\frac12)}$, there is a unique integer $i \in T^{(j)}$ and partial tableau $S^{(j)}$ of shape $\lambda_{j-\frac12}$ such that $T^{(j)}$ is obtained from $S$ by inserting $i$ via RSK. 
    Then  $i_j=i$ and    $T^{(j-\frac12)}$ is the standardization of $S^{(j)}$. 
    \item $T^{(j-1)}$ is obtained from $T^{(j-\frac12)}$ by adding $n$ to the square of $\lambda_j \setminus \lambda_{j-\frac12}$. 
\end{enumerate}
The initial tableau $T$ is just $T^{(0)}$.  This proves Identity \eqref{eq:Young1} 
for the case of even length.

For \vtx\ of odd length, 
given an element $(P, \Gamma)  \in \mathcal{SYT}(\lambda) \times \mathcal{VT}_{k}(\mu \rightarrow\lambda)$ for some $\lambda \vdash n$,
we can map it to the pair $(P', \Gamma')$,  where $P'$ is the SYT obtained from $P$ by removing the box containing $n$, and  $\Gamma'$ comes from appending $\lambda'$, the shape of $P'$,  to the end of $\Gamma$. 
Then $(P', \Gamma') \in \mathcal{SYT}(\lambda') \times \mathcal{VT}_{k+\frac12}(\mu \rightarrow\lambda')$  with  $\lambda' \vdash n-1$.  This, together with the bijection $\phi_n^k$ in \eqref{bijection:1}, gives a bijection from $\mathcal{SYT}(\mu) \times [n]^k$ to the disjoint union 
\[
\bigcupdot_{\lambda' \vdash n-1 } \mathcal{SYT}(\lambda') \times \mathcal{VT}_{k+\frac12}(\mu \rightarrow\lambda')
\]
and finishes the proof.  

\end{proof} 

\begin{example}  \label{ex:DI} 
Let  $n = 6$, $k=3$, and $\mu=(3,2,1)$.   We apply the map $\phi^3_6$ to the pair $(T, \bsy{i})$, where 
\[
T= \ytableausetup{smalltableaux} \begin{ytableau}
1 &2 &4  \\
3  & 6 \\
5
\end{ytableau}
\qquad \bsy{i}=(2,4, 3) 
\]
to obtain the $(T^{(j)})$ sequence as follows.
 \[
 \ytableausetup{smalltableaux}
 \begin{ytableau}
1 &2 &4  \\
3 & 6  \\
5
\end{ytableau} \ \to\ \begin{ytableau}
1 &2 &4 \\
3\\
5
\end{ytableau}\ \to\  \begin{ytableau}
1 & 2 & 5 \\
3\\
4\\
6
\end{ytableau}\ \to\  \begin{ytableau}
1  & 2  & 5\\
3  \\
4
\end{ytableau}\ \to\ \begin{ytableau}
1 &2  & 4\\
3 & 6\\
 5  
\end{ytableau}  \to\  \begin{ytableau}
1  & 2  & 4\\
3 \\
5
\end{ytableau}\ \to\ \begin{ytableau}
1 &2  & 3\\
4 & 5\\
 6  
\end{ytableau}  
\]

Thus, the image of $(T, \bsy{i})$ under $\phi_6^3$ is 
$(P, \Gamma)$, where  
\[ 
P= \ytableausetup{smalltableaux} \begin{ytableau}
1 &2 & 3  \\
4 & 5 \\
6
\end{ytableau},  \quad  
 \Gamma=  \ytableausetup{boxsize=7pt}  \left(\ydiagram{3,2,1},\ \ydiagram{3,1,1,}, \ \ydiagram{3,1,1,1},\ \ydiagram{3,1,1}, \ \ydiagram{3,2,1}, \ \ydiagram{3,1,1}, \ \ydiagram{3,2,1}
 \right). \] 
\end{example}

\medskip 

\begin{proof}[Proof of Identity \eqref{eq:Young2}] 
 We modify the map $\phi_n^k$. Denote by 
$\mathcal{VT}_k^{\uparrow}(\mu \rightarrow \lambda)$   the set of 
up-down \vtx \ from $\mu$ to $\lambda$ in $2k$ steps. 
We will define a bijiection $\widetilde{\phi}_n^k$ from the set 
$\mathcal{SYT}(\mu) \times [n+1]^{k}$ to 
\[
\bigcupdot_{\lambda \vdash n}  \mathcal{SYT}(\lambda) \times \mathcal{VT}^\uparrow_{k}(\mu \rightarrow\lambda), 
\]
which proves Identity \eqref{eq:Young2}. 
%the disjoin union of $  \mathcal{SYT}(\lambda) \times 
%\mathcal{VT}^\uparrow_k(\mu \rightarrow \lambda)$,  where $\lambda$ ranges over all shapes of size $n$. It then gives 
%a bijective proof of \eqref{eq:Young2}:
%\begin{eqnarray*} 
%\widetilde{DI}_n^k: \mathcal{SYT}(\mu) \times [n+1]^{k}   \xleftrightarrow{} 
%\bigcupdot_{\lambda \vdash n}  \mathcal{SYT}(\lambda) \times \mathcal{VT}^\uparrow_{k}(\mu \rightarrow\lambda).
%\end{eqnarray*}

%The bijection $\widetilde{DI}_n^k$ is constructed as follows. 
Let $(i_1, i_2, \ldots, i_{k}) \in [n+1]^k$ be an integer sequence of length $k$ and $T$ be a SYT of shape $\mu \vdash n$.  We define a sequence of SYTs iteratively: 
The $0$-th
tableau  $T^{(0)}$ is just $T$. 
For integers $j=1, \ldots, k$, 
\begin{enumerate}
    \item $T^{(j-\frac12)}$ is the SYT of size $n+1$ obtained from $T^{(j)}$ by lifted inserting $i_j$. 
    
    \item $T^{(j)}$ is the SYT of size $n$ obtained from $T^{j-\frac12}$ by removing the square containing $n+1$. 
\end{enumerate}
Let $P=T^{(k)}$. 
The image of $(T, (i_1, \dots, i_{k}))$ under $\widetilde{\phi}_n^k$ is 
given by the pair $(P, \Gamma)$, where $\Gamma$ is the sequence of  shapes of $(T^{(0)}, T^{(\frac12)}, T^{(1)}, \dots, T^{(k)})$. 

The reason that $\widetilde{\phi}_n^k$ is bijective is that we can recover the initial input $T$ and $(i_1, \dots, i_k)$ 
from $(P, \Gamma)$ step by step from $T^{(k)}=P$. For $j=k, k-1, \dots, 1$
\begin{enumerate}
    \item $T^{(j-\frac12)}$ can be obtained from $T^{(j)}$ by putting $n+1$ to the extra square in the shape of $T^{(j-\frac12)}$ but not in that of $T^{(j)}$. 
    
    \item The inverse of RSK can recover the integer $i_{j-1}$ and a unique tableau 
    $S^{(j-1)}$ such that $T^{(j-\frac12)}$ is obtained from $S^{(j-1)}$ by inserting $i_{j-1}$ via RSK. Let $T^{(j-1)}$ be the standardization of $S^{(j-1)}$. 
   
    \item The initial SYT $T$ of shape $\mu$ is just $T^{(0)}$. 
\end{enumerate}

For \vtx\ of odd length $2k-1$, given $(T, \bsy{i}) \in \mathcal{SYT}(\mu) \times [n+1]^k$,  we  follow the same construction of  $\widetilde{\phi}^k_n$ as described above for 
$j=1, 2, \dots, k$, except that we stop at  $T^{k-\frac12}$. 
Let $P'=T^{(k-\frac12)}$ and $\Gamma'$ be the sequence of shapes of $(T^{(0)}, T^{(\frac12)}, \dots T^{(k-\frac12)})$. 
Hence the set of images $(P', \Gamma')$  is the disjoint union 
\[
\bigcupdot_{\lambda \vdash n+1} \mathcal{SYT}(\lambda) \times \mathcal{VT}^\uparrow_{k-\frac12}(\mu \rightarrow\lambda), 
\]
as desired. 
%where $\mathcal{VT}_{k-\frac12}(\mu \rightarrow\lambda)$ is the set of up-down vacillating tableaux from $\mu$ to $\lambda$ of length $2k-1$.
\end{proof} 

%Add an example. 
\begin{example}  \label{ex:map2} 
Let  $n = 5$, $k=3$, and $\mu=(3,1,1)$.   We apply the map $\widetilde{\phi}^3_5$ to the pair $(T, \bsy{i})$, where 
\[
T= \ytableausetup{smalltableaux} \begin{ytableau}
1 &2 &4  \\
3  \\
5
\end{ytableau}
\qquad \bsy{i}=(2,4,3) 
\]
to obtain the $(T^{(j)})$ sequence as follows.
 \[
 \ytableausetup{smalltableaux}
 \begin{ytableau}
1 &  2& 4  \\
3  \\
5
\end{ytableau} \ \to\ \begin{ytableau}
1 &2 &5 \\
3\\
4\\
6
\end{ytableau}\ \to\  \begin{ytableau}
1 & 2 & 5 \\
3\\
4
\end{ytableau}\ \to\  \begin{ytableau}
1  & 2  & 4\\
3  & 6\\
5
\end{ytableau}\ \to\ \begin{ytableau}
1 &2  & 4\\
3 \\
5  
\end{ytableau} \to\  \begin{ytableau}
1  & 2  & 3\\
4  & 5\\
6
\end{ytableau}\ \to\ \begin{ytableau}
1 &2  & 3 \\
4 &  5  
\end{ytableau}
\] 
Thus, the image of $(T, \bsy{i})$ under $\widetilde{\phi}_5^3$ is 
$(P, \Gamma)$, where  
\[ 
P= \ytableausetup{smalltableaux} \begin{ytableau}
1 &2 & 3  \\
4 & 5 
\end{ytableau},  \quad  
 \Gamma=  \ytableausetup{boxsize=7pt}  \left(\ydiagram{3,1,1},\ \ydiagram{3,1,1,1}, \ \ydiagram{3,1,1},\ \ydiagram{3,2,1},\ 
 \ydiagram{3,1,1}, \ 
\ydiagram{3,2,1}, \ 
 \ydiagram{3,2}  \right). \] 
\end{example}

\section{Deletion-insertion in powers of Young's lattice}

In this section we extend the deletion-insertion algorithm of Section 3 
to  $\mathbb{Y}^r$ for $r \geq 2$ and present bijective proofs for 
Identities \eqref{eq:UD-even}-\eqref{eq:UD-odd} and \eqref{eq:DU-even}-\eqref{eq:DU-odd}. 

Elements in $\mathbb{Y}^r$ are $r$-tuples
$\Lambda= (\lambda^{(1)}, \dots, \lambda^{(r)})$ of integer partitions ordered coordinate-wise by inclusion. 
The minimal element $\hat 0$ of $\mathbb{Y}^r$ is $(\emptyset, \cdots, \emptyset)$, and $\Lambda$
is of rank $n$ if and only if $|\lambda^{(1)}| + \cdots + |\lambda^{(r)}|=n$. 
 In addition, 
$\Lambda=(\lambda^{(1)}, \dots, \lambda^{(r)})$ covers $\Theta=(\mu^{(1)}, \mu^{(2)}, \dots, \mu^{(r)})$  if and only if $\lambda^{(j)}=\mu^{(j)}$ for all but one $j \in [r]$,  and for  the exceptional index $j$,  $\lambda^{(j)} \supset \mu^{(j)}$ 
and $|\lambda^{(j)} / \mu^{(j)}|=1$. 

Assume $\Lambda=(\lambda^{(1)}, \dots, \lambda^{(r)})$ is of rank $n$. 
Let $a_i=|\lambda^{(i)}|$. Then 
a saturated chain from $\hat 0$ to $\Lambda$ can be represented by 
a sequence of partial tableaux $(T^{(1)}, \cdots, T^{(r)})$, where 
$T^{(i)}$ is of shape $\lambda^{(i)}$, and the set of all entries in $T^{(1)}, \cdots, T^{(r)}$ is exactly $[n]$.

Let $\mathcal{E}(\Lambda)$ be the set of saturated chains from $\hat 0$ to $\Lambda$. It follows that 
$e(\Lambda) =|\mathcal{E}(\Lambda)|= \binom{n}{a_1, \dots, a_r} \prod_i f^{\lambda^{(i)}}.$ 
Let $\Theta, \Lambda \in \mathbb{Y}^r$. 
Following the notations in Section 3, 
we again use $\mathcal{VT}_i(\Theta \rightarrow \Lambda)$ to represent the set of down-up alternating walks from $\Theta$ to $ \Lambda$  of length $2i$,  and $\mathcal{VT}^\uparrow_i(\Theta \rightarrow \Lambda)$for the set of up-down ones.

For a fixed positive integer $r$, 
let $r[n]=[n]\times [r]$. We interpret  each pair $(i, c) \in r[n]$ as the \emph{integer $i$ with color $c$}.

\begin{proof}[Proof of Identity \eqref{eq:UD-even} in $\mathbb{Y}^r$] \ 
Let $\Theta \in (\mathbb{Y}^r)_n$.  We construct the bijection $\phi_{r,n}^k$ between the following sets: 

\begin{eqnarray*} 
\phi_{r,n}^k: \mathcal{E}(\Theta) \times (r[n])^k  
\longrightarrow
\bigcupdot_{\Lambda \in \mathbb{Y}^r_n  } \mathcal{E}(\Lambda) \times \mathcal{VT}_{k}(\Theta \rightarrow\Lambda). 
\end{eqnarray*}

Let $\bsy{i} = (i_1, \dots, i_k) \in (r[n])^k$, where the value of $i_j$ is in $[n]$ and the color of $i_j$ is  $c(i_j)$.  Start with a sequence of partial tableaux
$\Gamma_0= (T^{(1)}, \dots T^{(r)}) \in \mathcal{E}(\Theta)$, 
we define $\Gamma_0, \Gamma_\frac12, \Gamma_1, \dots, \Gamma_k$ iteratively for $i=0, \frac12, 1, \dots, k$, where  each 
$\Gamma_i$ is an $r$-tuple  of partial tableaux.  
Explicitly, for integers $j=1, 2, \dots, k$,  
\begin{enumerate}
    \item $\Gamma_{j-\frac12}$ is obtained from $\Gamma_{j-1}$ by removing the square containing $n$.  

    The set of entries in tableaux of $\Gamma_{j-\frac12}$ is $[n-1]$.
   
    \item  Let $t=c(i_{j})$. Then $\Gamma_{j}$ is obtained from $\Gamma_{j-\frac12}$ by
      \begin{enumerate}
          \item  adding 1 to  every entry $i$ in the partial tableaux of $\Gamma_{j-\frac12}$   satisfying $i \geq i_{j}$,  
           \item   inserting $i_{j}$ to the $t$-th  tableau using RSK row insertion. 
   \end{enumerate}
    The set of entries in tableaux of $\Gamma_{j}$ is $[n]$.     
\end{enumerate}
The image of $\phi_{r,n}^k$ on $(\Gamma_0, \bsy{i})$  is the pair $(\Gamma_k, \Delta)$, where 
$\Gamma_k$ is in $\mathcal{E}(\Lambda)$ for some $\Lambda \in \mathbb{Y}^r$, and $\Delta$ is the shape-sequence of $(\Gamma_0, \Gamma_\frac12, \Gamma_1, \dots, \Gamma_k)$, which is a down-up alternating walk from $\Theta$ to $\Lambda$ in $\mathbb{Y}^r$. 
The map $\phi_{r,n}^k$ is a bijection because both the deletion procedure and RSK row insertion are invertible; consequently, one can recover the sequence $\Gamma_k, \Gamma_{k-\frac12}, \dots, \Gamma_0$ from $\Gamma_k$ and $\Delta$. 
The value of $i_j$ is recovered via inverse RSK insertion, while its color is determined by identifying the component whose shape  changes  from $\Gamma_{j-\frac12}$ to $\Gamma_{j}$. 
\end{proof} 

\begin{example} \label{ex:Y^r}   
Let $r=n=k=2$ and $\Theta=(1, 1)$. Then $\mathcal{E}(\Theta)$ has two elements,  
$\ytableausetup{smalltableaux} 
(\begin{ytableau}
1 \end{ytableau}, 
\begin{ytableau}
    2 
\end{ytableau}
)$ and 
$\ytableausetup{smalltableaux} 
(\begin{ytableau}
2 \end{ytableau}, 
\begin{ytableau}
    1 
\end{ytableau}
)$, and $(r[n])^k$ has 16 elements. 

To illustrate the bijection $\phi_{r,n}^k$, we show in 
Table \ref{table1} the correspondence between $\mathcal{E}(\Theta) \times (r[n])^k$ to
the pairs $(\Gamma_k, \Delta)$. Here we represent elements in $(r[n])^k$ as $(i,j) \in \{1,\bar 1, 2, \bar 2\}$,  where an integer $i$ means it is of color 1, and 
$\bar i$ means it is of color 2. 

\end{example}

\renewcommand{\arraystretch}{1.2}

\begin{table}[p] 
    \thispagestyle{empty} 
\small 
    
\begin{tabular}{|c|c|l|c|}
\hline
$\mathcal{E}(\Theta)$ & $(i,j) \in \{1,\bar 1, 2, \bar 2\}$ & vacillating tableau $\Delta$ &   $\Gamma_k \in \mathcal{E}(\Lambda)$ \\
\hline

\multirow{16}{*}{$\ytableausetup{smalltableaux} 
(\begin{ytableau}
1 \end{ytableau}, 
\begin{ytableau}
    2 
\end{ytableau}
)$}
&  
% row 1 
$(1,1)$   &  
$\ytableausetup{boxsize=7pt}  
(\ydiagram{1}, \ydiagram{1}) \rightarrow 
(\ydiagram{1}, \emptyset) \rightarrow 
(\ydiagram{1,1}, \emptyset) \rightarrow 
(\ydiagram{1},\emptyset) \rightarrow 
(\ydiagram{1,1}, \emptyset)$   
& $\ytableausetup{smalltableaux} (\begin{ytableau}
1 \\
2\end{ytableau}, 
\emptyset )$    
\\ \cline{2-4}
&  
%row 2 
$(1,\bar 1)$   & 
 $\ytableausetup{boxsize=7pt}  
(\ydiagram{1}, \ydiagram{1}) \rightarrow 
(\ydiagram{1}, \emptyset) \rightarrow 
(\ydiagram{1,1}, \emptyset) \rightarrow 
(\ydiagram{1},\emptyset) \rightarrow 
(\ydiagram{1}, \ydiagram{1})$  
& $\ytableausetup{smalltableaux} ( \begin{ytableau}
  2 
\end{ytableau}, 
 \begin{ytableau}
1
\end{ytableau}
)$  
\\ \cline{2-4}
& 
%row 3
$(1, 2)$   & 
$\ytableausetup{boxsize=7pt} 
(\ydiagram{1}, \ydiagram{1}) \rightarrow 
(\ydiagram{1}, \emptyset) \rightarrow 
(\ydiagram{1,1}, \emptyset) \rightarrow 
(\ydiagram{1},\emptyset) \rightarrow 
(\ydiagram{2}, \emptyset)$  
& $\ytableausetup{smalltableaux} (\begin{ytableau}
  1 & 2
\end{ytableau}, 
\emptyset)$  
\\ \cline{2-4}
& 
%row 4 
$(1,\bar 2)$  & 
$\ytableausetup{boxsize=7pt}  
(\ydiagram{1}, \ydiagram{1}) \rightarrow 
(\ydiagram{1}, \emptyset) \rightarrow 
(\ydiagram{1,1}, \emptyset) \rightarrow 
(\ydiagram{1},\emptyset) \rightarrow 
(\ydiagram{1}, \ydiagram{1})$   
& $\ytableausetup{smalltableaux} (\begin{ytableau}
1 \end{ytableau}, 
\begin{ytableau}
    2
\end{ytableau}
) $ 
\\ \cline{2-4}
&   
%row 5
$(\bar 1, 1)$  &
$\ytableausetup{boxsize=7pt} 
(\ydiagram{1}, \ydiagram{1}) \rightarrow 
(\ydiagram{1}, \emptyset) \rightarrow 
(\ydiagram{1}, \ydiagram{1}) \rightarrow
(\emptyset, \ydiagram{1}) \rightarrow
(\ydiagram{1}, \ydiagram{1}
)$   
& $\ytableausetup{smalltableaux} (\begin{ytableau}
1 \end{ytableau}, 
\begin{ytableau}
    2 
\end{ytableau})$ 
 \\ \cline{2-4}
& 
%row 6
$(\bar 1, \bar 1)$   & 
$\ytableausetup{boxsize=7pt} 
(\ydiagram{1}, \ydiagram{1}) \rightarrow 
(\ydiagram{1}, \emptyset) \rightarrow 
(\ydiagram{1}, \ydiagram{1} )\rightarrow
(\emptyset, \ydiagram{1}) \rightarrow
(\emptyset, \ydiagram{1,1}
)$   
& $\ytableausetup{smalltableaux} (\emptyset, \begin{ytableau}
1 \\
    2 
\end{ytableau})$ 
\\ \cline{2-4}
&  
%row 7
$(\bar 1, 2)$     & 
$\ytableausetup{boxsize=7pt}  
(\ydiagram{1}, \ydiagram{1}) \rightarrow 
(\ydiagram{1}, \emptyset) \rightarrow 
(\ydiagram{1}, \ydiagram{1}) \rightarrow
(\emptyset, \ydiagram{1}) \rightarrow
(\ydiagram{1}, \ydiagram{1}
)$   
& $\ytableausetup{smalltableaux} (\begin{ytableau}
2 \end{ytableau}, 
\begin{ytableau}
    1
\end{ytableau})$ 
\\ \cline{2-4}
& 
%row 8 
$(\bar 1, \bar 2)$   & 
$\ytableausetup{boxsize=7pt} 
(\ydiagram{1}, \ydiagram{1}) \rightarrow 
(\ydiagram{1}, \emptyset) \rightarrow 
(\ydiagram{1}, \ydiagram{1}) \rightarrow
(\emptyset, \ydiagram{1}) \rightarrow
(\emptyset, \ydiagram{2})$   
& $\ytableausetup{smalltableaux} (\emptyset, \begin{ytableau}
1 & 2 \end{ytableau}
)$ 
 \\ \cline{2-4}
& 
%row 9
$(2, 1)$   & 
$\ytableausetup{boxsize=7pt}  (\ydiagram{1}, \ydiagram{1}) \rightarrow ( \ydiagram{1}, \emptyset) \rightarrow 
( \ydiagram{2}, \emptyset) \rightarrow (\ydiagram{1},\emptyset) \rightarrow 
(\ydiagram{1,1}, \emptyset)$   
& $\ytableausetup{smalltableaux} (\begin{ytableau}
1\\ 2 \end{ytableau}, 
\emptyset)$ 
\\ \cline{2-4}
& 
%row 10
$(2, \bar 1)$  & 
$\ytableausetup{boxsize=7pt}  (\ydiagram{1}, \ydiagram{1}) \rightarrow ( \ydiagram{1}, \emptyset) \rightarrow 
( \ydiagram{2}, \emptyset) \rightarrow (\ydiagram{1},\emptyset) \rightarrow 
(\ydiagram{1}, \ydiagram{1})$   
& $\ytableausetup{smalltableaux} (\begin{ytableau}
2 \end{ytableau}, 
\begin{ytableau}
1 \end{ytableau}
)$ 
 \\ \cline{2-4}
&
%row 11
$(2, 2)$   & 
$\ytableausetup{boxsize=7pt}  (\ydiagram{1}, \ydiagram{1}) \rightarrow ( \ydiagram{1}, \emptyset) \rightarrow 
( \ydiagram{2}, \emptyset) \rightarrow (\ydiagram{1},\emptyset) \rightarrow 
(\ydiagram{2}, \emptyset)$   
& $\ytableausetup{smalltableaux} (\begin{ytableau}
1 & 2 \end{ytableau}, 
\emptyset)$ 
 \\ \cline{2-4}
& 
%row 12
$(2, \bar 2)$  & 
$\ytableausetup{boxsize=7pt}  (\ydiagram{1}, \ydiagram{1}) \rightarrow ( \ydiagram{1}, \emptyset) \rightarrow 
( \ydiagram{2}, \emptyset) \rightarrow (\ydiagram{1},\emptyset) \rightarrow 
(\ydiagram{1}, \ydiagram{1})$   
& $\ytableausetup{smalltableaux} (\begin{ytableau}
1 \end{ytableau}, 
\begin{ytableau}
2 \end{ytableau}
)$ 
 \\ \cline{2-4}
& 
%row 13
$(\bar 2, 1)$  & 
$\ytableausetup{boxsize=7pt}  (\ydiagram{1}, \ydiagram{1}) \rightarrow 
( \ydiagram{1}, \emptyset) \rightarrow 
( \ydiagram{1}, \ydiagram{1}) \rightarrow 
(\ydiagram{1}, \emptyset) \rightarrow 
(\ydiagram{1,1}, \emptyset)$   
& $\ytableausetup{smalltableaux} (\begin{ytableau}
1 \\ 2 \end{ytableau}, 
\emptyset
)$ 
 \\ \cline{2-4}
& 
%rw 14
$(\bar 2, \bar 1)$  & 
$\ytableausetup{boxsize=7pt}  (\ydiagram{1}, \ydiagram{1}) \rightarrow 
( \ydiagram{1}, \emptyset) \rightarrow 
( \ydiagram{1}, \ydiagram{1}) \rightarrow 
(\ydiagram{1},\emptyset) \rightarrow 
(\ydiagram{1}, \ydiagram{1})$   
& $\ytableausetup{smalltableaux} (\begin{ytableau}
    2 \end{ytableau}, 
,\begin{ytableau}
1 \end{ytableau}
)$ 
\\ \cline{2-4}
& 
%row 15
 $(\bar 2, 2)$  & 
 $\ytableausetup{boxsize=7pt}  (\ydiagram{1}, \ydiagram{1}) \rightarrow ( \ydiagram{1}, \emptyset) \rightarrow 
( \ydiagram{1}, \ydiagram{1}) \rightarrow (\ydiagram{1},\emptyset) \rightarrow 
(\ydiagram{2}, \emptyset)$   
& $\ytableausetup{smalltableaux} (\begin{ytableau}
1  & 2 \end{ytableau}, \emptyset 
)$ 
 \\ \cline{2-4}
& 
%row 16 
$(\bar 2, \bar 2)$  &
$\ytableausetup{boxsize=7pt}  (\ydiagram{1}, \ydiagram{1}) \rightarrow ( \ydiagram{1}, \emptyset) \rightarrow 
( \ydiagram{1}, \ydiagram{1}) \rightarrow (\ydiagram{1}, \emptyset) \rightarrow 
(\ydiagram{1}, \ydiagram{1})$   
& $\ytableausetup{smalltableaux} (\begin{ytableau}
1 \end{ytableau}, 
\begin{ytableau}
2 \end{ytableau}
)$ 
 \\ \hline

\multirow{16}{*}{$\ytableausetup{smalltableaux} 
(\begin{ytableau}
2 \end{ytableau}, 
\begin{ytableau}
    1
\end{ytableau}
)$}
% Row1
&  $(1,1)$   & 
$\ytableausetup{boxsize=7pt} 
(\ydiagram{1}, \ydiagram{1}) \rightarrow
( \emptyset, \ydiagram{1} ) \rightarrow 
( \ydiagram{1},\ydiagram{1}) \rightarrow
(\ydiagram{1}, \emptyset) \rightarrow 
(\ydiagram{1,1}, \emptyset)$   
& $\ytableausetup{smalltableaux} (\begin{ytableau}
1 \\2 \end{ytableau}, \emptyset 
)$ 
\\ \cline{2-4}
%Row 2
&   $(1,\bar 1)$   & 
$\ytableausetup{boxsize=7pt}  (\ydiagram{1}, \ydiagram{1}) \rightarrow ( \emptyset, \ydiagram{1}) \rightarrow 
( \ydiagram{1}, \ydiagram{1}) \rightarrow (\ydiagram{1},\emptyset) \rightarrow 
(\ydiagram{1}, \ydiagram{1})$   
& $\ytableausetup{smalltableaux} (\begin{ytableau}
2 \end{ytableau}, 
\begin{ytableau}
1 \end{ytableau}
)$ 
\\ \cline{2-4}
%Row 3 
&  $(1, 2)$   & 
$\ytableausetup{boxsize=7pt}  (\ydiagram{1}, \ydiagram{1}) \rightarrow ( \emptyset, \ydiagram{1}) \rightarrow 
( \ydiagram{1}, \ydiagram{1}) \rightarrow
(\ydiagram{1}, \emptyset) \rightarrow 
(\ydiagram{2}, \emptyset)$   
& $\ytableausetup{smalltableaux} (\begin{ytableau}
1 & 2 \end{ytableau}, \emptyset 
)$ 
\\ \cline{2-4}
%Row 4
&  $(1,\bar 2)$  & 
$\ytableausetup{boxsize=7pt}  (\ydiagram{1}, \ydiagram{1}) \rightarrow ( \emptyset, \ydiagram{1}) \rightarrow 
( \ydiagram{1}, \ydiagram{1}) \rightarrow (\ydiagram{1},\emptyset) \rightarrow 
(\ydiagram{1}, \ydiagram{1})$   
& $\ytableausetup{smalltableaux} (\begin{ytableau}
1 \end{ytableau}, 
\begin{ytableau}
2 \end{ytableau}
)$ 
\\ \cline{2-4}
%Row5 
&   $(\bar 1, 1)$  & 
$\ytableausetup{boxsize=7pt}  (\ydiagram{1}, \ydiagram{1}) \rightarrow (  \emptyset, \ydiagram{1}) \rightarrow 
( \emptyset, \ydiagram{1,1}) \rightarrow 
(\emptyset, \ydiagram{1}) \rightarrow 
(\ydiagram{1}, \ydiagram{1})$   
& $\ytableausetup{smalltableaux} (\begin{ytableau}
1  \end{ytableau}, 
\begin{ytableau}
    2
\end{ytableau}
)$ 
\\ \cline{2-4}
%row 6 
&  $(\bar 1, \bar 1)$ &
$\ytableausetup{boxsize=7pt}  (\ydiagram{1}, \ydiagram{1}) \rightarrow 
( \emptyset, \ydiagram{1}) \rightarrow 
( \emptyset, \ydiagram{1,1}) \rightarrow 
(\emptyset, \ydiagram{1}) \rightarrow 
( \emptyset, \ydiagram{1,1}) 
)$   
& $\ytableausetup{smalltableaux} (
\emptyset, 
\begin{ytableau}
1\\ 2 \end{ytableau} 
)$ 
\\ \cline{2-4}
%Row 7

&  $(\bar 1, 2)$     & 
$\ytableausetup{boxsize=7pt}  (\ydiagram{1}, \ydiagram{1}) \rightarrow ( \emptyset, \ydiagram{1}) \rightarrow 
( \emptyset, \ydiagram{1,1}) \rightarrow 
(\emptyset, \ydiagram{1}) \rightarrow 
(\ydiagram{1}, \ydiagram{1})$   
& $\ytableausetup{smalltableaux} (\begin{ytableau}
2 \end{ytableau}, 
\begin{ytableau}
1 \end{ytableau}
)$ 
\\ \cline{2-4}
%Row 8 
&  $(\bar 1, \bar 2)$   & 
$\ytableausetup{boxsize=7pt}  (\ydiagram{1}, \ydiagram{1}) \rightarrow ( \emptyset, \ydiagram{1}) \rightarrow 
( \emptyset, \ydiagram{1,1}) \rightarrow (\emptyset, \ydiagram{1}) \rightarrow 
(\emptyset, \ydiagram{2})$   
& $\ytableausetup{smalltableaux} (\emptyset, \begin{ytableau}
1 & 2 \end{ytableau}
)$ 
 \\ \cline{2-4}

%Row 9
& $(2, 1)$   & 
$\ytableausetup{boxsize=7pt}  (\ydiagram{1}, \ydiagram{1}) \rightarrow 
(\emptyset, \ydiagram{1}) \rightarrow 
(\ydiagram{1}, \ydiagram{1}) \rightarrow 
(\emptyset, \ydiagram{1}) \rightarrow 
(\ydiagram{1}, \ydiagram{1})$   
& $\ytableausetup{smalltableaux} (\begin{ytableau}
1 
\end{ytableau}, 
\begin{ytableau}
    2
\end{ytableau}
)$    
 \\ \cline{2-4}

%Row 10 
& $(2, \bar 1)$  & 
$\ytableausetup{boxsize=7pt}  (\ydiagram{1}, \ydiagram{1}) \rightarrow (\emptyset, \ydiagram{1}) \rightarrow 
(\ydiagram{1}, \ydiagram{1}) \rightarrow 
(\emptyset, \ydiagram{1}) \rightarrow 
(\emptyset, \ydiagram{1,1})$   
& $\ytableausetup{smalltableaux} (
\emptyset, \begin{ytableau}
1 \\ 2
\end{ytableau}
)$    
\\ \cline{2-4}

%row 11
& $(2, 2)$   & 
$\ytableausetup{boxsize=7pt}  
(\ydiagram{1}, \ydiagram{1}) \rightarrow 
(\emptyset, \ydiagram{1}) \rightarrow 
(\ydiagram{1}, \ydiagram{1}) \rightarrow 
(\emptyset, \ydiagram{1}) \rightarrow 
(\ydiagram{1}, \ydiagram{1})$   
& $\ytableausetup{smalltableaux} (\begin{ytableau}
2
\end{ytableau}, 
\begin{ytableau}
1
\end{ytableau}
)$    
\\ \cline{2-4}

%row 12
& $(2, \bar 2)$  & 
$\ytableausetup{boxsize=7pt}  
(\ydiagram{1}, \ydiagram{1}) \rightarrow 
(\emptyset, \ydiagram{1}) \rightarrow 
(\ydiagram{1}, \ydiagram{1}) \rightarrow 
(\emptyset, \ydiagram{1}) \rightarrow 
(\emptyset, \ydiagram{2})$   
& $\ytableausetup{smalltableaux} 
(\emptyset, 
\begin{ytableau}
1 & 2 
\end{ytableau}
)$    
\\ \cline{2-4}

%row 13
& $(\bar 2, 1)$  & 
$\ytableausetup{boxsize=7pt}  
(\ydiagram{1}, \ydiagram{1}) \rightarrow 
(\emptyset, \ydiagram{1}) \rightarrow 
(\emptyset, \ydiagram{2}) \rightarrow 
(\emptyset, \ydiagram{1}) \rightarrow 
(\ydiagram{1}, \ydiagram{1})$   
& $\ytableausetup{smalltableaux} (\begin{ytableau}
1
\end{ytableau}, 
\begin{ytableau}
2
\end{ytableau}
)$    
\\ \cline{2-4}

% row 14 
& $(\bar 2, \bar 1)$  & 
$\ytableausetup{boxsize=7pt}  
(\ydiagram{1}, \ydiagram{1}) \rightarrow 
(\emptyset, \ydiagram{1}) \rightarrow 
(\emptyset, \ydiagram{2}) \rightarrow 
(\emptyset, \ydiagram{1}) \rightarrow 
(\emptyset, \ydiagram{1,1})$   
& $\ytableausetup{smalltableaux} (
\emptyset,
\begin{ytableau}
1 \\ 2
\end{ytableau} 
)$    
\\ \cline{2-4}

%row 15
& $(\bar 2, 2)$  & 
$\ytableausetup{boxsize=7pt}  
(\ydiagram{1}, \ydiagram{1}) \rightarrow 
(\emptyset, \ydiagram{1}) \rightarrow 
(\emptyset, \ydiagram{2}) \rightarrow 
(\emptyset, \ydiagram{1}) \rightarrow 
(\ydiagram{1}, \ydiagram{1})$   
& $\ytableausetup{smalltableaux} (\begin{ytableau}
2
\end{ytableau}, 
\begin{ytableau}
1
\end{ytableau}
)$    
\\ \cline{2-4}

%row16 
& $(\bar 2, \bar 2)$  & 
$\ytableausetup{boxsize=7pt}  
(\ydiagram{1}, \ydiagram{1}) \rightarrow 
(\emptyset, \ydiagram{1}) \rightarrow 
(\emptyset, \ydiagram{2}) \rightarrow 
(\emptyset, \ydiagram{1}) \rightarrow 
(\emptyset, \ydiagram{2})$   
& $\ytableausetup{smalltableaux} (
\emptyset, 
\begin{ytableau}
1 & 2 
\end{ytableau} 
)$    

\\ \hline

\end{tabular}
\caption{Correspondence of $\phi_{r,n}^k$ for $r=n=k=2$. } \label{table1}
\end{table}

\normalsize 

\begin{proof}[Proof of Identity \eqref{eq:UD-odd} in $\mathbb{Y}^r$] 
Observe that the following two sets are identical, as both are sets of $D^n (UD)^k$-walks from $\Theta \in (\mathbb{Y}^r)_n$ to $\hat 0$: 
\begin{eqnarray} \label{bijection:r2}
\bigcupdot_{\Lambda \in \mathbb{Y}^r_{n}  } \mathcal{E}(\Lambda) \times \mathcal{VT}_{k}(\Theta \rightarrow\Lambda)
 = 
 \bigcupdot_{\Lambda \in \mathbb{Y}^r_{n-1}  } \mathcal{E}(\Lambda) \times \mathcal{VT}_{k+\frac12}(\Theta \rightarrow\Lambda). 
\end{eqnarray}
Hence $\phi_{r,n}^k$ also induces a bijecive proof for \eqref{eq:UD-odd}. 
\end{proof}

\begin{proof}[Proof of Identity \eqref{eq:DU-even} in $\mathbb{Y}^r$] 
Given $\Theta \in (\mathbb{Y}^r)_n$ and  $\bsy{i} \in (r[n+1])^k$,   we construct a bijection $\widetilde{\phi}_{r,n}^k$ between the following sets: 

\begin{eqnarray} \label{bijection:r3}
 \mathcal{E}(\Theta) \times (r[n+1])^k  
\longrightarrow
\bigcupdot_{\Lambda \in \mathbb{Y}^r_{n}  }   \mathcal{E}(\Lambda) \times \mathcal{VT}^\uparrow_{k}(\Theta \rightarrow\Lambda). 
\end{eqnarray}

Assume $\bsy{i} =(i_1, \dots, i_{k}) \in (r[n+1])^k$, where the value of $i_j$ is in $[n+1]$ and the color of $i_j$ is  $c(i_j)$, 
and $\Gamma = (T^{(1)}, \dots, T^{(r)}) \in \mathcal{E}(\Theta)$, we define $\Gamma_0, \Gamma_\frac12, \Gamma_1, \dots, \Gamma_k$ iteratively, where each $\Gamma_i$ is an $r$-tuple of partial tableaux. First, let $\Gamma_0=\Gamma$.

For $j=1, 2, \dotsm k$, 
\begin{enumerate}
    \item Let $t=c(i_j)$.  Then $\Gamma_{j-\frac12}$ is obtained from $\Gamma_{j-1}$ by 
    \begin{enumerate}
        \item adding 1 to every entry $i$ in the partial tableaux of $\Gamma_{j-1}$ satisfying $i \geq i_j$,  and 
        \item inserting $i_j$ to the $t$-th tableau using RSK row insertion.  
    \end{enumerate}
    
    The set of entries in tableaux of $\Gamma_{j-\frac12}$ is $[n+1]$.  

    \item  $\Gamma_j$ is obtained from $\Gamma_{j-\frac12}$  by removing the square containing $n+1$.   

    The set of entries in tableaux of $\Gamma_j$ is $[n]$.    
\end{enumerate}

 %$\Gamma_k \in \mathcal{E}(\Lambda)$ for some $\Lambda \in \mathbb{Y}^r_n$. 

The image of $(\Gamma, \bsy{i})$ under $\widetilde{\phi}_{r,n}^k$ is the pair $(\Gamma_k, \Delta)$, where $\Delta$ is the shape-sequence of $(\Gamma_0, \Gamma_\frac12, \Gamma_1, \dots, \Gamma_k)$.  It is a bijection because every step is invertible. We skip the details here. 
\end{proof}

\begin{proof}[Proof of Identity \eqref{eq:DU-odd} in $\mathbb{Y}^r$] 
Use the same construction as of  $\widetilde{\phi}_{r,n}^k$ in the preceding proof, 
except that one stops at $\Gamma_{k-\frac12}$. The final image is the pair $(\Gamma_{k-\frac12}, \Delta')$, where $\Delta'$ is the shape-sequence of $(\Gamma_0, \Gamma_\frac12, \Gamma_1, \dots, \Gamma_{k-\frac12})$. 
 \end{proof}

\section{Final Remarks} 
\noindent 
%\textsc{Final Remark.}\ 

We finish this paper with some  remarks. 

\subsection{The deletion procedure }\   

\iffalse 
one can also use Fomin's general framework of growth diagrams \cite{Fomin95-2}, which is a generalization of the Robinson-Schensted correspondence, to construct bijective proofs for Identities 
\eqref{eq:Young1} - \eqref{eq:Young3} in $\mathbb{Y}$. 
 The detailed description was given  by Kratthenthaler  \cite{Kra23}  for  Identity \eqref{Eq:Id2}  and other identities involving numbers of \vtx \ between certain pairs of $\mu$ and $\lambda$.  The bijection using growth diagram is different from Halverson and Lewandowski's one based on the deletion-insertion algorithm.  
\fi

The main ingredients of the bijections in this paper are the deletion procedure and 
(lifted) RSK row insertion.   RSK insertion has the useful property that if a tableau $S$ is obtained from a partial tableau $T$ by inserting $i$, then $i$ and $T$ can be recovered uniquely from $S$  together with the shape of $T$. Therefore, for the deletion procedure, it suffices to have an operation with the following property: whenever a tableau $T$ is obtained from a tableau $S$ by deleting $i$, one can recover $S$ uniquely from $T$, $i$, and the shape of $S$.

The deletion procedure used in this paper is the simplest possible choice: delete the maximal entry of $S$.

 The  bijection $\phi_n^k$  constructed in the proof of Identity \eqref{eq:Young1} can also be described in terms of  growth diagrams.  It is essentially equivalent to the bijection given by Krattenthaler in \cite[Section 4]{Kra23}. The only difference is that  the integer sequence $\bsy{i}$  used in \cite{Kra23} appears in  reversed order 
 relative to the sequence associated with $\phi_n^k$. 
 For example, let $\mu=(3,2,1)$ and $\lambda=(2,2,1,1)$.   
 Following \cite[Figure 5]{Kra23},  
 take the followng SYT $P$ of shape $\lambda$ and  vacillating tableau  $\Gamma$ from $\mu$ to $\lambda$\footnote{There is a typo in \cite{Kra23}. The fifth partition in $\Gamma$ should be $(2,2,2)$, instead of $(2,2,1,1)$}: 
 \[ 
P= \ytableausetup{smalltableaux} \begin{ytableau}
1 &2  \\
3 & 5  \\
4 \\
6
\end{ytableau},  \quad  
 \Gamma=  \ytableausetup{boxsize=7pt}  \left(\ydiagram{3,2,1},\ \ydiagram{2,2,1,}, \ \ydiagram{2,2,2},\ \ydiagram{2,2,1}, \ \ydiagram{2,2,2}, \ \ydiagram{2,2,1}, \ \ydiagram{2,2,1,1}
 \right). \] 
 
 Under our map $\phi_n^k$ the pair $(P, \Gamma) \in \mathcal{SYT}(\lambda) \times \mathcal{VT}_3(\mu \rightarrow \lambda)$ corresponds to the pair 
 $(T, (3,2,2))$, while the bijection in \cite{Kra23} maps $(P, \Gamma)$ to $(T, (2,2,3))$, 
 where 
 \[
 T= \ytableausetup{smalltableaux} \begin{ytableau}
1 & 4 & 6   \\
2 & 5  \\
3  
\end{ytableau}.  
 \]
Our bijection $\phi_n^k$ provides a more concise description of  the growth-diagram bijection in \cite{Kra23}. 
Krattenthaler considered permutation fillings of a Ferrers diagram with $n+k$ rows whose lengths, from top to bottom, are   $n, n+1, \dots, n+k, \dots, n+k$, where $n+k$ occurs $n$ times. In fact, it suffices to consider only the first $k$ rows, whose lengths are $n, n+1, \dots, n+k-1$.  See Figure~\ref{fig:growth} for an illustration. In the figure we keep the same notations as in \cite{Kra23}  and align the squares for the growth diagram  at the lower-left corner. 

Label all the vertices along the left boundary  with empty partitions. Then   
the bijection $\phi_n^k$ asserts that the following two sets of data  determine the same growth diagram:
\begin{enumerate}
    \item The sequence of partitions along the top  boundary, corresponding to a tableau $P \in \mathcal{SYT}(\lambda)$, together with the sequence of partitions along the right boundary, corresponding to  a vacillating tableau $\Gamma$ from $\lambda$ to $\mu$; 
    \item A restricted filling of the Ferrers diagram, corresponding to  a sequence
    $\bsy{i} \in [n]^k$,  and the sequence of distinct partitions along the bottom boundary, corresponding to a tableau $T \in \mathcal{SYT}(\mu)$. 
    The filling  is restricted in the sense that each row contains exactly one cross and no two crosses occupy  the same column. Moreover, 
    two consecutive partitions along the bottom boundary differ by one square if and only if the column immediately above them contains no cross. 
\end{enumerate}

\begin{figure}[htp]
    \centering
    
\begin{tikzpicture}[scale=.9] 
\foreach \y/\n in {2/6,1/7,0/8}
  \foreach \x in {1,...,\n}
    \draw (\x-1,\y) rectangle ++(1,1);

   \draw (8,0) -- (9,0); 

   % Empty-set symbols
\node at (-0.2,-.3) {$\emptyset$};
\node at (-0.2,1) {$\emptyset$};
\node at (-0.2,2) {$\emptyset$};
\node at (-0.2,3.2) {$\emptyset$};

% Put \lambda and \mu 
\node at (10, -.3) {$321 = \mu$ }; 
\node at (6.7, 3.2) {$2211 = \lambda$ }; 

% put the cross 
\node at (1.5, 2.5) {X}; 
\node at (2.5, 1.5) {X};
\node at (4.5, 0.5) {X}; 

% put P, T, and \Gamma 
\node at (3.5, 3.7) {-- P -- }; 
\node at (4.5, -1) {-- T --} ; 
\node at (10, 2) { $\Gamma$  }; 
\node at (9.7, 2.4) {$\setminus$ }; 
\node at (10.2, 1.6) {$\setminus$};

% SYT P 
\node at (1,3.2) {1}; 
\node at (2,3.2) {2}; 
\node at (3, 3.2) {21}; 
\node at (4, 3.2) {211}; 
\node at (5, 3.2) {221}; 

% VT from lambda to mu
\node at (6.2, 2.2) {221}; 
\node at (7.4, 2.2) {222}; 
\node at (7.4, 1.2) {221}; 
\node at (8.4, 1.2) {222}; 
\node at (8.2, -.3) {221};

% SYT T to \mu 
\node at (1, -.3) {1}; 
\node at (2, -.3) {1}; 
\node at (3, -.3) {1};
\node at (4, -.3) {11}; 
\node at (5, -.3) {11}; 
\node at (6, -.3) {111}; 
\node at (7, -.3) {211}; 
%\node at (8, -.3) {221};   did above in VT 

\end{tikzpicture}

\caption{The growth diagram for bijection $\phi_n^k$ with $n=6$ and $k=3$.}
    \label{fig:growth}
\end{figure}
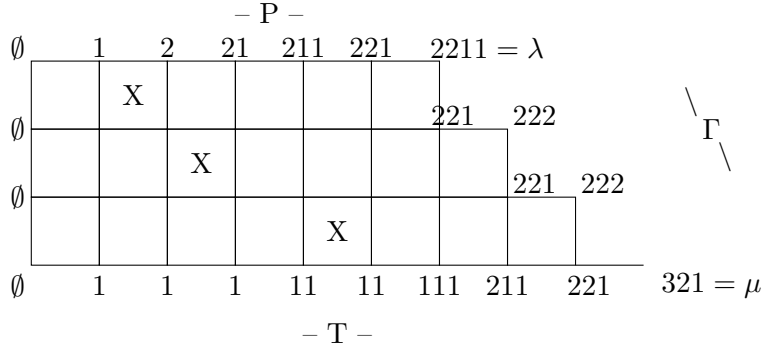

There are other choices for the deletion procedure.  Halverson and Lewandowski used jeu de taquin to remove the entry $i_j$ in the $j$-th step of the algorithm. Another natural alternative is to delete the minimal element, or the  $t$-th element for a prefixed $t$ in $[n]$, in a partial tableau again using jeu de taquin. It would be interesting to understand the combinatorial significance of
these and other deletion procedures, and to investigate how different choices affect the resulting bijections.

\subsection{Rim Hook Lattice}\ 
 Another model of $r$-differential posets is the $r$-rim hook lattice. 
A rim hook is a set of squares  forming a contiguous strip with at most one square on each diagonal. An $r$-rim hook is a rim hook containing exactly $r$ squares. 

For two integer partitions $\mu$ and $\lambda$, if $\mu \subset \lambda$ and $\lambda \setminus \mu$ is an $r$-rim hook, then  we say that $\mu \lessdot \lambda$. Equivalently, $\mu \lessdot \lambda$ if $\mu$ can be obtained from $\lambda$ by removing an $r$-rim hook.  We define $\mu \preceq \lambda$ if $\mu$ can be obtained from $\lambda$ by successively removing $r$-rim hooks. 

Let $RH_r$ be the set of  integer partitions $\lambda$ such that
\[
\emptyset \preceq \lambda.
\]
Fomin and Stanton \cite{FS98} proved that 
$RH_r$ is a lattice and is isomorphic to $\mathbb{Y}^r$. 
Hence the bijections in Section 4 naturally extend to $RH_r$. 

Nevertheless, we can consider alternating walks directly on the Hasse diagram of $RH_r$ and construct deletion-insertion algorithms  using a rim hook version of the jeu de taquin together with  the rim hook insertion algorithm; both are developed by Stanton and White \cite{White83,StWhite85}. 

To get a set with cardinality $rn$, instead of integers in $[n]$ with $r$ colors, 
one uses the set of \emph{hook tableaux of size $r$ and content $i \in [n]$}.  Here, 
a hook tableau of size $r$ is a tableau of shape $(i, 1^{r-i})$ whose  entries are all equal, and this common value is called the content.
The rim hook insertion algorithm developed by Stanton and White allows one to insert a hook tableau of size $r$ with content $i$ into an $r$-rim hook tableau $P$ 
with no entry $i$. This algorithm can be inverted, in a similar way that one inverts RSK
row insertion.

Combinatorially,  walks in $\mathbb{Y}^r$ and $RH_r$ lead to two different representations of $r$-colored discrete structures, 
therefore provide distinct perspectives on their combinatorial properties. For example,  Chen and Guo \cite{ChenGuo} studied colored matchings via oscillating rim hook tableaux and oscillating tableaux in $\mathbb{Y}^2$, that is, walks from $\hat 0$ to $\hat 0$  consisting of up and down steps that are not necessarily alternating.

\subsection{Representation Interpretations} \ 
 Identity \eqref{eq: Identity1} arises from the representation theory of partition algebras $\mathbb{C}A_k(n)$, reflecting the Schur-Weyl duality between the symmetric group $S_n$ and $\mathbb{C}A_k(n)$ on the $k$-fold tensor power $M_n^{\otimes k}$ of the permutation module $M_n$ of $S_n$; see \cite[Sec.2.1]{HL05}.  This duality has been  further developed in \cite{BH19, BHH17}. Identities \eqref{eq:Young1} and \eqref{eq:Young3} exhibit a strong similarity to Identity \eqref{eq: Identity1}, and rim hook tableaux are closely related to the wreath product $S_m[Z_k]$, in particular,  the hyperoctahedral group (the signed permutation group). 
 Hence 
it would be interesting to provide representation-theoretic 
interpretations for the new identities proved in the present paper.

\iffalse 
Can one give Identities \eqref{eq:Young1} --\eqref{eq:Young3} 
a representation interpretation?  

Paper \cite{HL05} would be a starting point. Also check the references it cites, as well as papers citing it for more on representation of partition algebra. 
\fi

\bibliographystyle{abbrv}
\bibliography{Reference}

\end{document}